\documentclass [11pt,oneside,a4paper,mathscr]{amsart}
\usepackage{amscd,amsmath,amssymb,euscript}
\usepackage[frame,cmtip,curve,arrow,matrix,line,graph]{xy}
\usepackage{tikz-cd}
\usepackage{bigints}
\usetikzlibrary{matrix}
\usepackage{hyperref}
\usepackage{soul}
\usepackage{amssymb, amsmath, amsthm}
\usepackage{amssymb}
\usepackage{graphicx}
\usepackage{stmaryrd}

\theoremstyle{plain}
\newtheorem{theorem}{Theorem}[section]

\theoremstyle{definition}

\newtheorem{lemma}[theorem]{Proposition}

\numberwithin{equation}{section}

\begin{document}

\title{Groups of order $p^3$ are mixed Tate}
\author{Tudor P\u adurariu}
\address{Department of Mathematics, Columbia University, 
2990 Broadway, New York, NY 10027}
\email{tpadurariu@gmail.com}

\date{}
\maketitle
\begin{abstract}
Let $G$ be a finite group. A natural place to study the Chow ring of the classifying space $BG$ is Voevodsky's triangulated category of motives, inside which Morel--Voevodsky and Totaro have defined motives $M(BG)$ and $M^c(BG)$, respectively. We show that, for any group G of order $p^3$ over a field of characteristic not $p$ which contains a primitive $p^3$-th root of unity, the motive $M(BG)$ is a mixed Tate motive. We also show that, for a finite group $G$ over a field of characteristic zero, $M(BG)$ is a mixed Tate motive if and only $M^c(BG)$ is a mixed Tate motive.
\end{abstract}

\renewcommand{\thefootnote}{\fnsymbol{footnote}} 
\footnotetext{\emph{$2010$ Mathematics Subject Classification:} Primary 14C15, Secondary 18E30, 20D05.

}     
\renewcommand{\thefootnote}{\arabic{footnote}} 

\section{Introduction}

\subsection{Mixed Tate groups} The group cohomology of a group $G$ can be computed as the cohomology (with twisted coefficients) of the classifying space $BG$. One would like to understand what part of the group cohomology of $G$ comes from algebraic geometry. 
Morel--Voevodsky \cite{mor} and Totaro \cite{to} defined the motive of a classifying space $M(BG)$ and the motive of a classifying space with compact supports $M^c(BG)$, respectively, as objects in $\text{DM}(k;R)$, Voevodsky's ``big" triangulated category of motives over the field $k$ with coefficients in a commutative ring $R$ \cite{voe}. One can recover the motivic (co)homology groups of $BG$ as defined by Edidin--Graham \cite{edi} by computing the motivic (co)homology groups of these motives. 

Inside $\text{DM}(k;R)$, one can define the subcategory of mixed Tate motives $\text{DMT}(k; R)$ as the smallest triangulated
and closed  under arbitrary direct sums subcategory which contains all the objects $R(j)$ with $j\in \mathbb{Z}$. 
We prove in Theorem \ref{thm5} that the motive $M(BG)$ is \textit{mixed Tate} if and only if $M^c(BG)$ is mixed Tate for a finite group $G$. We will simply say that a finite group $G$ is mixed Tate if $M^c(BG)$ is in the category $\text{DMT}(k; R)$. From now on, we will restrict the discussion in the Introduction to finite groups. 
Our main result is:

\begin{theorem}\label{thm}
Let $G$ be a group of order $p^3$ and let $k$ be a field of characteristic not $p$ which contains a primitive $p^3$-root of unity. 
Then $M^c(BG)$ is mixed Tate.
\end{theorem}

One is interested in understanding $p$-groups because one recovers important information about a given finite group by studying its Sylow groups. The precise form of this philosophy which is applicable in our case is \cite[Lemma 9.3]{to} which says that $BG$ is mixed Tate with $\mathbb{Z}/p$ or $\mathbb{Z}_{(p)}$ coefficients if $BH$ is, where $H$ is a $p$-Sylow subgroup of $G$.

\subsection{Other properties of finite groups}
A group $G$ is called \textit{stably rational} if it has a faithful representation $V$ such that $V\sslash G$ is stably rational over $\mathbb{C}$. A group $G$ has \textit{the weak Chow--K\"{u}nneth} property if $CH^{*} \left(BG\right)\twoheadrightarrow CH^{*} \left(BG_E\right)$ is surjective for every extension of fields $E/k$. 
If $G$ is mixed Tate, then $BG$ is stably rational, satisfies the weak Chow--Kunneth property, and has \textit{trivial unramified cohomology}, see \cite[Section 9]{to} for definitions and references. We do not know whether any of these properties of a finite group $G$ are equivalent.


\subsection{Related results}
In all the following examples, we assume that $k$ is a field in which $p$ is invertible and which contains $|G|$-roots of unity, where $G$ is the group studied.

The starting point for studying these properties of a group $G$ are Bogomolov's \cite{bo} and Saltman's \cite{s} examples of groups of order $p^7$ and $p^9$, respectively, which are not stably rational. 
Chu--Kang, and Chu et. al. \cite{chu}, \cite{chuhu} showed that for every $p$-group $G$ of order $\leq p^4$ or $2$-group of order $\leq 2^5$ and for every $G$-representation $V$, the quotient $V\sslash G$ is rational. This property is stronger than saying that $BG$ is stably rational.

Bogomolov \cite{bo} showed (with a further correction in \cite{hoshi}) that every $p$-group of order $\leq p^4$, for $p$ odd prime, or $\leq 2^5$ for $p$ equal to $2$, has trivial unramified cohomology, and that these are the best possible bounds.

Totaro \cite[Section 10]{to} showed that all $2$-groups of order $\leq 2^5$ and all $p$-groups of order $\leq p^4$ have the weak Chow-K\"{u}nneth property. He also showed \cite[Corollary 9.10]{to} that all abelian $p$-groups are mixed Tate. There are groups of order $p^5$ for $p$ odd which do not have the weak Chow-K\"{u}nneth property \cite[Discussion after Corollary 3.1]{to} and thus which are not mixed Tate.

In view of these examples, it is worth investigating whether all $p$-groups of order $\leq p^4$ and all $2$-groups of order $\leq 2^5$ are actually mixed Tate. Our methods only apply to $p$-groups of order $\leq p^3$ and to some groups of order $p^4$ as explained in Section \ref{4}.

\subsection{Structure of the paper}
		In Section \ref{2}, we recall the definitions of linear schemes and of the motives $M(X)$ and $M^c(X)$ for a quotient stack $X$ in $\text{DM}(k; R)$. In Section \ref{3}, we reduce the proof of Theorem \ref{thm} to Theorem \ref{41} and we prove three technical preliminary Propositions. 
		Section \ref{4} contains the proof of Theorem \ref{41} which says that for a group $G$ of order $p^3$ and $V$ an irreducible $G$-representation of dimension $p$, the scheme $V\sslash G$ is a linear scheme. The proof is inspired by a result of
		Chu--Kang \cite{chu} that says that $V\sslash G$ is rational for $G$ of order $p^3$ and $V$ a $G$-representation.
		In Section \ref{5}, we show that $M(BG)$ is mixed Tate if and only if $M^c(BG)$ is mixed Tate.
		
\subsection{Acknowledgements}
The paper was written in the scholar year 2014-2015 while I was a senior at UCLA.
	I would like to thank Burt Totaro for suggesting the problem and for advising me throughout the writing of the paper. I thank the referees for numerous useful comments on previous drafts of the paper that improved the exposition.

\section{Definitions and notations}\label{2}
	
\subsection{}	Fix $p$ a prime number.
	Unless otherwise stated, we will denote by $k$ a field of characteristic not $p$ which contains a primitive $p^2$-root of unity. In Section \ref{5}, we assume that the characteristic of $k$ is zero.
	
	All the schemes considered will be separated schemes of finite type over $k$. 
	One can define the Chow groups $CH_i(X)$ as the group of dimension $i$ algebraic cycles modulo rational equivalence \cite{fu}. One can further define the higher Chow groups \cite{bl}, or the motivic (co)homology groups of such a scheme \cite{voe}, see \cite[Section 5]{to} for a brief overview of these topics.
	
	Let $A$ be an affine $k$-scheme with a linear action of a reductive group $G$. We denote by $A\sslash G:=\text{Spec}\left(\mathcal{O}_A^G\right)$ the quotient scheme and by $A/G$ the corresponding quotient stack.
	
	For a finite group $G$, we denote by $|G|$ the order of $G$. We denote by $[n]$ the set $\{1,\cdots, n\}$.

\subsection{}\label{dmdef}	
	We will work in the category $\text{DM}(k;R)$, the ``big" triangulated category of motives over the field $k$ with coefficients in the commutative ring $R$ \cite[Section 5]{to}, see also the general references \cite{maz}, \cite{voe}. 
	
	The exponential characteristic of $k$ is $1$ if $k$ has characteristic zero and $p$ if $k$ has characteristic $p>0$.
	We will assume throughout the paper that the exponential characteristic of $k$ is invertible in $R$.
Voevodsky defined two natural functors from the category of schemes to $\text{DM}(k;R)$, which we will write as $M$ and $M^c$ \cite{voe}, see also \cite[Section 5]{to}. 

	We can associate a motive to any quotient stack $X=Y/G$, with $Y$ a quasi-projective scheme over $k$ and $G$ an affine group scheme of finite type over $k$ such that there is a $G$-equivariant ample line bundle on $Y$, as follows \cite[Section 8]{to}. Choose $G$-representations $V_1\hookrightarrow V_2\hookrightarrow \dots$ of $G$ such that $\text{codim}\,(S_i\text{ in }V_i)$ increases to infinity, where $S_i$ is the locus of $V_i$ where $G$ does not act freely. Denote by $M_i(X):=M\left(\left((V_i-S_i)\times Y\right)/G\right)$ and define
	\[M(X)=\text{hocolim}\left(\dots\leftarrow M_2(X)\leftarrow M_1(X)\right),\]
where the maps are induced by the inclusions $V_i\hookrightarrow V_{i+1}$. To define $M^c(X)$, choose $G-$representations $\dots \twoheadrightarrow V^{2}\twoheadrightarrow V^1$ with loci $S^i$ with the same property as above. Let $M^c_i(X):=M_c\left(\left((V^i-S^i)\times Y\right)/G\right)$. Let $n_i$ be the rank of the bundle $V^i$. Define
 \[M^c(X)=\text{holim}\left(\dots\rightarrow M^c_2(X)(-n_2)[-2n_2]\rightarrow M^c_1(X)(-n_1)[-2n_1]\right),\]
where the maps are induced by the projections $V^{i+1}\twoheadrightarrow V^i$. The definitions of $M^c(X)$ and $M(X)$ are independent of the choices of $V_i$ and $V^i$, see \cite[Theorem 8.4 and the discussion in Section 8]{to}.

\subsection{}\label{strati}	
\textit{A linear scheme} over $k$ is defined inductively as follows \cite[Section 5, page 17-18]{to}: all the affine spaces are linear; if $Z\subset X$ is closed, and $X$ and $Z$ are linear, then $X\setminus Z$ is linear; further, if $X\setminus Z$ and $Z$ are linear, then $X$ is linear \cite[page 17]{to}. There are examples of schemes with mixed Tate motive, but which are not linear schemes \cite{gu}.

Let $X$ be a linear scheme over $k$ and let $R$ be a ring whose exponential characteristic is invertible in $R$. Then $M^c(X)$ is a mixed Tate motive.
\smallskip

Let $I$ be a finite set, let $X_i\subset X$ be locally closed irreducible subschemes of $X$, and let $d=\dim(X)$. For $e\leq d$, let $Y_e$ be the union of $X_i$ for $i\in I$ such that $\dim(X_i)=e$.
We say that $X$ has a \textit{stratification} $(X_i)_{i\in I}$ if there is a partition of underlying topological spaces
\[X=\bigsqcup_{i\in I}X_i\]
and $Y_e$ is open in $X\setminus \bigsqcup_{f>e}Y_f=\bigsqcup_{g\leq e}Y_g$ for every $e\leq d$.

\section{The plan of the proof and preliminaries}\label{3}

\subsection{} Theorem \ref{thm} is known for abelian groups \cite[Corollary 9.10]{to}.
The two non-abelian groups of order $8$ are the dihedral and the quaternion group. Theorem \ref{thm} holds for them by \cite[Corollary 9.7]{to}.
It thus suffices to show the following:

\begin{theorem}\label{31}
Let $p$ be an odd prime,
let $k$ be a field of characteristic not $p$ which contains a primitive $p^2$-root of unity,
and let $G$ be a non-abelian group of order $p^3$. Then $M^c(BG)$ is mixed Tate.
\end{theorem}

There are sufficient conditions on $G$ which imply that $G$ is mixed Tate. For example, by \cite[Theorem 9.6]{to} it is enough to show that every proper subgroup $H\subset G$ is mixed Tate and that there exists a faithful representation $V$ of $G$ such that the variety $(V-S)\sslash G$ is mixed Tate, where $S$ is the closed subset of $V$ where $G$ does not act freely. 

For $K\subset G$ a subgroup, let $N_K:=\{g\in G|\,gKg^{-1}=K\}$ be the normalizer of $K$ and let $N'_K:=N_K/K$.

\begin{lemma}\label{linearschemes}
Let $G$ be a finite group such that 
$N'_K$ is abelian for every subgroup $1<K\subset G$.
Let $V$ be a representation of $G$ and let $S\subset V$ be the locus of points with non-trivial stabilizer. Then $(V-S)\sslash G$ is a linear scheme if and only if $V\sslash G$ is a linear scheme.
\end{lemma}

\begin{proof}
It suffices to check that $S\sslash G$ is a linear scheme. We use induction on $|G|$. The statement is clear if $|G|$ is a prime number, because then $G$ is a cyclic group and $S$ is a subspace of $V$, and so $S\sslash G\cong S$ is an affine space.

For $K\subset G$ a subgroup, let $V^K\subset V$ be the subspace of points fixed by $K$ and let 
\[V_K:=V^K-\bigcup_{K<L\subset G}V^L.\] 
If $K'$ is a subgroup of $G$ conjugate to $K$, the images of $V_K\sslash N'_K$ and $V_{K'}\sslash N'_{K'}$ in $V\sslash G$ are the same. 
Let $I$ be a set of subgroups of $G$ such that any subgroup $K$ of $G$ is conjugate to a unique group in $I$.
We have that $S=\bigsqcup_{1<K\subset G}V_K$ and there is a stratification
\[S\sslash G=\bigsqcup_I V_K\sslash N'_K\] It suffices to check that $V_K\sslash N'_K$ is a linear scheme for any $1<K\subset G$. 
The group $N'_K$ is abelian, so it satisfies the hypothesis of the Proposition. We have that $|N'_K|<|G|$, 
so by the induction hypothesis we know that $V_K\sslash N'_K$ is a linear scheme if and only if $V^K\sslash N'_K$ is a linear scheme.
By Proposition \ref{32}, the quotient $V^K\sslash N'_K$ is a linear scheme, thus $V_K\sslash N'_K$ is a linear scheme.
\end{proof}



Any non-abelian group of order $p^3$ has a faithful irreducible representation. Indeed, a $p$-group has a faithful irreducible representation if and only if its center is cyclic \cite[p.29]{is}, and $Z(G)$ has order $p$ for any non-abelian group of order $p^3$. Moreover, all irreducible representations of a group $G$ of order $\leq p^4$ have dimension $1$ or $p$. Any group of order $p^3$ satisfies the hypothesis of Proposition \ref{linearschemes} because
for every subgroup $1<K\subset G$, the quotient
$N_K/K$ has order $1, p,$ or $p^2$, and thus it is abelian.
It is thus sufficient to prove the following:

\begin{theorem}\label{41}
Let $k$ be a field of characteristic not $p$ which contains a primitive $p^2$-root of unity. 
Let $G$ be a non-abelian group of order $p^3$ and let $V$ be an irreducible representation of degree $p$. Then $V\sslash G$ is a linear scheme.
\end{theorem}

\subsection{} There are two non-abelian groups of order $p^3$. For a classification of $p$-groups of order $\leq p^4$ and their representations, see \cite{chu}.

\subsubsection{}\label{first}
The first group is $G\cong (\mathbb{Z}/p\times \mathbb{Z}/p)\rtimes \mathbb{Z}/p$, which can be also written as $$G=\left\langle \sigma, \pi, \tau| \sigma^p=\pi^p=\tau^p=1, \sigma\pi=\pi\sigma, \sigma\tau=\tau\sigma, \tau\pi\tau^{-1}=\sigma\pi\right\rangle.$$ It has a faithful irreducible representation $(\rho, V)$ which can be written explicitly on a basis $(e_i)_{i=1}^p$ of $V$ as follows: \begin{align*}
\rho(\sigma)&=\text{diag}(\zeta,\dots,\zeta),\\
\rho(\pi)&=\text{diag}(1,\zeta,\dots,\zeta^{p-1}),\\
\rho(\tau)&=P,
\end{align*} where $P$ is the matrix which permutes the basis $e_1\mapsto e_2\mapsto\dots\mapsto e_p\mapsto e_1$, and $\zeta$ is a primitive $p$-th root of unity. 

\subsubsection{}\label{second}
The second group is $G\cong \mathbb{Z}/p^2\rtimes \mathbb{Z}/p$, which can be also written as $$G=\left\langle \sigma, \tau| \sigma^{p^2}=\tau^p=1,\tau\sigma \tau^{-1}=\sigma^{1+p}\right\rangle.$$ It has a faithful  irreducible representation $(\rho, V)$ given by 
\begin{align*}
    \rho(\sigma)&=\text{diag}\,(\omega,\omega^{1+p},\dots,\omega^{1+p(p-1)}),\\
\rho(\tau)&=P,
\end{align*}
where $\omega$ is a primitive $p^2$-root of unity and $P$ is the permutation matrix defined above. 

\subsection{}
The proof of Theorem \ref{41} will be given in Section \ref{4}. In the rest of this Section, we include proofs for two Propositions used in its proof. The first one gives a proof of the already known fact that $BG$ is mixed Tate for $G$ abelian group \cite[Corollary 9.10]{to}. Recall that the exponent of a group is defined as the least common multiple of the orders of all elements of the group.

\begin{lemma}\label{32}
Let $N$ be an abelian $p$-group, and let $V$ be a $N$-representation over a field $k$ of characteristic not $p$ which contains the $p^e$ roots of unity, where $p^e$ is the exponent of $N$. Then $\text{Spec }k[V]^N$ is a linear scheme.
\end{lemma}

\begin{proof}
As $\text{char } k \neq p$, the representation $V$ decomposes as a sum of one dimensional representations, and thus we can choose a basis $x_1,\dots,x_d$ of $V$ on which $N$ acts diagonally. We prove the statement by induction on $|N|$. 
The base case, when $N$ is the trivial group, is clear. In general, choose $\sigma \in N$ such that $N=\left\langle \sigma\right\rangle\oplus M$, where $\left\langle \sigma\right\rangle$ denotes the subgroup of $N$ generated by $\sigma$. Assume that $\sigma$ has order $p^s$.
We will use the following stratification:
$$ 
\text{Spec }k[x_1,\dots,x_d]=\bigsqcup_{J\subset [d]} 
\text{ Spec }k\left[x_j^{\pm 1}\Big| j\in J\right],$$
where the disjoint union is taken after all sets $J\subset [d]$. This stratification is the partition of the affine space $\mathbb{A}^d_k$ into $2^d$ schemes $P_J$ with $x_j\neq 0$ for $j\in J$ and $x_j=0$ for $j\not \in J$. 
We obtain a stratification 
\begin{equation}\label{strati2}
\text{Spec }k[x_1,\dots,x_d]^{\left\langle \sigma\right\rangle}=\bigsqcup_{J\subset [d]} 
\text{ Spec }k\left[x_j^{\pm 1}\Big| j\in J\right]^{\left\langle \sigma\right\rangle}.\end{equation} 
It is enough to show that 
\begin{equation}\label{monomials}
\text{Spec }k\left[x_1^{\pm 1},\dots,x_d^{\pm 1}\right]^{\left\langle \sigma\right\rangle}\cong
\text{ Spec } k\left[y_j^{\pm 1}\right],\end{equation} where the $y_j$ are monomials in $x_i$. The analogous statement holds for any stratum on the right hand side of \eqref{strati2}.
Once we show \eqref{monomials}, we can reduce the problem from $N$ to $M$ for various representations of $M$.

To find such a decomposition, let $\sigma\cdot x_i=\zeta^{a_i}x_i,$ where $\zeta$ is a primitive $p^s$-root of unity chosen such that $a_1=1$. Then $$k\left[x_1^{\pm 1},\dots,x_d^{\pm 1}\right]^\sigma=k\left[x_1^{p^s},x_2x_1^{-a_2},\dots,x_dx_1^{-a_d},\frac{1}{x_1^Qx_2\dots x_d}\right],$$ where $Q:=p^s-a_2-\dots-a_d.$ The right hand side is included in the left hand side, and $k\left[x_1^{\pm 1},\dots,x_d^{\pm 1}\right]$ is a free $k\left[x_1^{p^s},x_2x_1^{-a_2},\dots,x_dx_1^{-a_d},\frac{1}{x_1^Qx_2\dots x_d}\right]$-module of rank $p^s$, so the two sides are indeed equal.
\end{proof}

Consider the torus $\left(\mathbb{G}_m\right)^p$ with coordinates $w_1,\cdots, w_p$ and let $W\subset \left(\mathbb{G}_m\right)^p$ be the subtorus with $w_1\cdots w_p=1$. The action of the cyclic group $\mathbb{Z}/p$ of order $p$ which permutes the factors of $\left(\mathbb{G}_m\right)^p$ by $w_i\mapsto w_{i+1}$ for $1\leq i\leq p$, where $w_{p+1}:=w_1$, extends to an action of $\mathbb{Z}/p$ on $W$.

\begin{lemma}\label{33}
The schemes $S:=W\sslash \mathbb{Z}/p$ and $T:=\left(\left(\mathbb{G}_m\right)^p-W\right)\sslash \mathbb{Z}/p$ are linear schemes.
\end{lemma}

\begin{proof}
Let $\tau$ be a generator of the cyclic group $\mathbb{Z}/p$. 
 Define \[W_d=1+\zeta^d w_1+\dots+\zeta^{d(p-1)}w_1\dots w_{p-1}\] for $d=0,\ldots,p-1$.
The stratification we are going to use is \[S=\bigsqcup_{d=0}^{p-1}\space \,S_d,\] where the schemes $S_d$ are defined as 
\[S_d:=\text{Spec }\left(k\left[w_1^{\pm 1},\ldots,w_{p-1}^{\pm 1},\frac{1}{W_d}\right]\Big/\left(W_0,\dots,W_{d-1}\right)\right)^\tau.\]
We will show that every such piece is a linear scheme. 

\textbf{Step 1.} We first explain the argument for $S_0=\text{Spec }k\left[w_1^{\pm 1},\cdots,w_{p-1}^{\pm 1},\frac{1}{W_0}\right]^\tau$. 
Define \[s_i:=\frac{\prod_{j\leq i}w_j}{W_0},\] for $i\in \{0,\ldots,p-1\}$, $w_0:=1$. Observe that $s_0+\dots+s_{p-1}=1$ and that $k\left[w_1^{\pm 1},\cdots, w_{p-1}^{\pm 1},\frac{1}{W_0}\right]\cong k\left[s_0^{\pm 1},\cdots, s_{p-1}^{\pm 1}\right]\big/(s_0+\cdots+s_{p-1}-1)$. Further, $\tau$ acts via $\tau: s_0\mapsto s_1\mapsto \dots\mapsto s_{p-1}\mapsto s_0.$
To show that $$\text{Spec }\left(k\left[s_0^{\pm 1},\cdots, s_{p-1}^{\pm 1}\right]/\left(s_0+\dots+s_{p-1}-1\right)\right)^\tau$$ is a linear scheme, we linearize the action by introducing the variables \[v_i=s_0+\zeta^is_1+\dots+\zeta^{i(p-1)}s_{p-1},\] $v_0=1$. Then $\tau v_i=\zeta^{-i}v_i$ and \[s_i=\frac{v_0+\zeta^{-i}v_1+\dots+\zeta^{-i(p-1)}v_{p-1}}{p}.\]
In this basis, $S_0$ becomes $$\text{Spec}\left(k\left[v_0,\dots,v_{p-1},\frac{1}{\prod_{i=0}^{p-1}\tau^i(l)}\right]/(v_0-1)\right)^\tau \cong \text{Spec}\,k\left[v_1,\dots,v_{p-1},\frac{1}{\prod_{i=0}^{p-1}\tau^i(l)}\right]^\tau,$$ where $l=1+v_1+\dots+v_{p-1}$ is the equation of a hyperplane. 
Now, we can realize $S_0$ as the complement of a linear scheme inside an affine space. Indeed, 
\begin{multline*}
    \text{Spec }k[v_1,\dots,v_{p-1}]=\\\text{Spec}\,k\left[v_1,\dots,v_{p-1},\frac{1}{\prod_{i=0}^{p-1}\tau^i(l)}\right] \bigsqcup \text{Spec }\left(k\left[v_1,\dots,v_{p-1}\right]/\prod_{i=0}^{p-1}\tau^i(l)\right)
    \end{multline*}
    and $\tau$ acts on both terms on the bottom line. 
    
    Observe that $\text{Spec }\left(k\left[v_1,\dots,v_{p-1}\right]/\prod_{i=0}^{p-1}\tau^i(l)\right)$ is the union of the hyperplanes $l$, $\tau(l), \cdots, \tau^{p-1}(l)$, which are cyclically permuted by $\tau$. 
Both $\text{Spec }k[v_1,\dots,v_{p-1}]^\tau$ and $\text{Spec }\left(k\left[v_1,\dots,v_{p-1}\right]/\prod_{i=0}^{p-1}\tau^i(l)\right)^\tau$ are linear schemes, so $S_0$ is indeed a linear scheme. 
\\

\textbf{Step 2.}
Fix $0\leq d\leq p-1$.
The proof that $S_d$ is a linear scheme is similar to the one in Step $1$. 
Define 
\[s_i=\frac{\prod_{j\leq i}w_j}{W_d},\] for $i=0,\ldots,p-1$, $w_0:=1$. Observe that $s_0+\dots+\zeta^{d(p-1)}s_{p-1}=1$ and that
$k\left[w_1^{\pm 1},\cdots, w_{p-1}^{\pm 1},\frac{1}{W_d}\right]=
k\left[s_0^{\pm 1},\cdots, s_{p-1}^{\pm 1}\right]/\left(s_0+\cdots+\zeta^{d(p-1)}s_{p-1}-1\right)$. Furthermore, we have that
\[W_e=\frac{s_0+\dots+\zeta^{e(p-1)}s_{p-1}}{s_0}\] for $e\leq d$,
so computations similar to those for $S_0$ show that
$$S_d\cong \text{Spec} \left(k\left[s_0^{\pm 1},\cdots, s_{p-1}^{\pm 1}\right]/I\right)^\tau,$$
where $I$ is the ideal generated by $s_0+\zeta^es_1+\dots+\zeta^{e(p-1)}s_{p-1}$ for all $0\leq e\leq d-1$, and by $s_0+\zeta^ds_1+\dots+\zeta^{d(p-1)}s_{p-1}-1$.
Changing the basis to $v_j$ defined as in Step 1, we find out that $$S_d\cong \text{Spec }\left(k\left[v_{d+1},\dots,v_{p-1}, \frac{1}{\prod_{i=0}^{p-1}\tau^i(l)}\right]^\tau\right).$$ The end of the argument in Step 1 shows that $S_d$ is a linear scheme. 
\\

\textbf{Step 3.}
The proof that $T$ is a linear scheme is already contained in the above argument. Indeed, introduce the basis \[v_j=s_0+\zeta^js_1+\dots+\zeta^{j(p-1)}s_{p-1},\] for $j=0,\ldots, p-1$. Then we need to show that 
$$\text{Spec }\left(k\left[v_0,\dots,v_{p-1},\frac{1}{\prod_{i=0}^{p-1}\tau^i(l)}\right]^\tau\right)$$ is a linear scheme, where $\tau$ acts on the $v_i$ by $\tau(v_i)=\zeta^{-i}v_i$ and $l=v_0+\dots+v_{p-1}$ is a hyperplane. The same argument as in Step $1$ shows this is a linear scheme.
\end{proof}

\section{Proof of Theorem \ref{41}}\label{4}

\subsection{}
In the beginning, we will work in a little more general framework which also covers some groups of order $p^4$. 
Thus, assume for the moment that $G$ has order $\leq p^4$ and has an irreducible representation of dimension $p$. We may assume that $V$ is faithful, and let $\rho:G\rightarrow GL(V)$. As $\rho$ is irreducible, it is induced from a one dimensional representation of a subgroup $N\subset G$, that is, $\rho=\text{Ind}_N^G\psi$ with $\psi:N\rightarrow GL(W)$ and with $W$ one dimensional \cite{la}. As $V$ has dimension $p$, the subgroup $N$ has index $p$ in $G$, and so $N\trianglelefteq G$. 

Choose representatives $\left\{1, t,\dots,t^{p-1}\right\}$ for the cosets of $G/N$. The explicit form of $\rho$ is $$\rho(g)=(\psi(t^{-i}gt^j))_{0\leq i,j\leq p-1},$$ where $\psi(g)=0$ if $g \not \in N$. 

If $Z(G)\not \subset N$, we can choose $t\in Z(G)$. Then $\rho(g)=(\psi(gt^{i-j})),$ so $\rho(g)=\psi(g)I,$ for every $g\in N$. As $\rho$ is faithful, this implies that $N\subset Z(G)$, and further that $G$ is abelian, contradicting that $G$ has an irreducible representation of dimension $p$. 

We thus have that $Z(G)\subset N$. In order for $\rho$ to be faithful, $\psi|_{Z(G)}$ needs to be faithful, too, so $Z(G)$ is cyclic.

Using the explicit description of $\rho$, we have that $\rho(G)\subset T\cdot W,$ where $T$ is the group of diagonal matrices and $W$ is the group of permutation matrices. By identifying $G$ with its image $\rho(G)$, $G$ can be written as a semi-direct product $N\rtimes M$, with $M\cong \mathbb{Z}/p,$ and $N$ an abelian $p$-group with $|N|\leq p^3.$

\subsection{}
The plan is to construct a decomposition of $V\sslash G$ into smaller linear schemes. We isolate one open subset of $V\sslash G$ and decompose its complement in linear schemes. After that, we show that the open subset is itself a linear scheme.

Choose a basis $x_1,\dots,x_p$ of $V$ on which $N$ acts diagonally and which is cyclically permuted by $\tau$, the generator of $M$. Observe that $$V\sslash G=\text{Spec }k[x_1,\dots,x_p]^G=\text{Spec }\left(k[x_1,\dots,x_p]^N\right)^\tau.$$
As we have already discussed in the proof of Proposition \ref{32}, there is a stratification
$$ 
\text{Spec }k[x_1,\dots,x_p]=\bigsqcup_{J\subset [p]} 
\text{ Spec }k\left[x_j^{\pm 1}\big| j\in J\right],$$
where the disjoint union is taken after all sets $J\subset [p]$. This stratification is the partition of the affine space $\mathbb{A}^p_k$ in the $2^p$ schemes $P_J$ with $x_j\neq 0$ for $j\in J$ and $x_j=0$ for $j\not \in J$. 
As $N$ acts linearly on the functions $x_i$ for $1\leq i\leq p$, we have that
$$\text{Spec }k[x_1,\dots,x_p]^N=\text{Spec }k\left[x_1^{\pm 1},\dots,x_p^{\pm 1}\right]^N\sqcup \bigsqcup_{J<[p]}\, \text{Spec }k\left[x_j^{\pm 1}|\,j\in J\right]^N.$$
By Proposition \ref{32}, each $\text{Spec }k\left[x_j^{\pm 1}|\,j\in J\right]^N$ for $J\subset [p]$ is a linear scheme. 

Let $t:[p]\to [p]$ be the function $t(x)=x+1$ for $x\leq p-1$ and $t(p)=1$. For $J\subset [p]$, let $t(J):=\{t(x)|\,x\in J\}\subset [p]$.  
Observe that $\tau$ permutes the schemes $S_J=\text{Spec }k\left[x_j^{\pm 1}|\,j\in J\right]^N$ by sending $S_J$ to $S_{t(J)}$. Consequently, there is a stratification

$$\text{Spec }k[x_1,\dots,x_p]^G=\text{Spec }k\left[x_1^{\pm 1},\dots,x_p^{\pm 1}\right]^G\sqcup \bigsqcup_A \text{Spec }S_J,$$ where $A$ is a set of representatives of the equivalence classes of the action of $t$ on the set of proper subsets of $[p]$. This means that, in order to show that $V\sslash G$ is a linear scheme, we have to prove that $\text{Spec }k\left[x_1^{\pm 1},\dots,x_p^{\pm 1}\right]^G$ is a linear scheme. We do this in the next Subsection.

\subsection{}
The study of this open piece is inspired by \cite{chu}. We begin by analyzing the $Z(G)$-invariants. If we can conveniently reduce the dimension of the scheme
$\text{Spec }k\left[x_1^{\pm 1},\dots,x_p^{\pm 1}\right]$
on which $G$ acts from $p$ to $p-1$, for example by finding a $G$-invariant element among the $Z(G)$-invariants, then the resulting ring will give a natural $\mathbb{Z}[\tau]$-representation on $\mathbb{Z}^{p-1}$. This representation was shown in \cite[page 687]{chu} to be generated by one element. By a theorem of Reiner \cite{re}, this representation is the canonical representation of $\mathbb{Z}[\tau]$ on $\mathbb{Z}[\zeta]$. This reduction can be done for the group $G\cong (\mathbb{Z}/p\times \mathbb{Z}/p)\rtimes \mathbb{Z}/p$. 

If all the elements of $N$ act by the same character of the $Z(G)$-invariants, then we can make a change of variables to reduce to the case of $\text{Spec }k\left[w_1^{\pm 1},\dots, w_p^{\pm 1}\right]^\tau,$ where $\tau$ cyclically permutes the $w_i$. For example, this is the case for $G\cong (\mathbb{Z}/p^2)\rtimes \mathbb{Z}/p$. In both situations, the final ingredient will be Proposition \ref{33}. 

\subsubsection{}\label{432} Assume that $G$ has order $p^3$. Then $Z(G)$ acts on $V$ via multiples of the identity, so

$$k\left[x_1^{\pm 1},\dots,x_p^{\pm 1}\right]^{Z(G)}=
k\left[x_1^{p},x_1^{-p},\frac{x_2}{x_1},\dots,\frac{x_1}{x_p}\right]=k\left[y_2^{\pm 1},\dots,y_p^{\pm 1}\right]\left[y_1^{\pm 1}\right],$$ for $y_1=x_1^{p}$, $y_i=\frac{x_{i+1}}{x_i}$, $i=2,\ldots,p$. Assume that we can replace $y_1$ with a $G$-invariant monomial $z_1$ such that 
$$k\left[y_2^{\pm 1},\dots,y_p^{\pm 1}\right]\left[y_1^{\pm 1}\right]=k\left[y_2^{\pm 1},\dots,y_p^{\pm 1}\right]\left[z_1^{\pm 1}\right].$$ 
This can be done when $G\cong (\mathbb{Z}/p\times \mathbb{Z}/p)\rtimes \mathbb{Z}/p$. Recall the notations from Subsection \ref{first}.
Indeed, in this case $Z(G)=\left\langle \sigma\right\rangle$. For the representation $(\rho, V)$ described in Subsection \ref{first}, 
$\pi$ acts on any $y_i$, $i=2,\ldots,p$, by multiplication with $\zeta$ and it fixes $y_1$, while \[\tau: y_2\mapsto \dots\mapsto y_p\mapsto \frac{1}{y_2\dots y_p}\] and $\tau(y_1)=y_1y_2^p$. If we replace $y_1$ by $z_1=y_1y_2^{p-1}\dots y_{p-1}^2y_p$, then $z_1$ is indeed $G$-invariant and $$k\left[y_2^{\pm 1},\dots,y_p^{\pm 1}\right]\left[y_1^{\pm 1}\right]=k\left[y_2^{\pm 1},\dots,y_p^{\pm 1}\right]\left[z_1^{\pm 1}\right].$$
Even more, the same argument works for a $p$-group of cardinal $p^4$ with $Z(G)\cong \mathbb{Z}/p^2$ and $N$ different from $\mathbb{Z}/p^3$. Indeed, in this case, $N\cong Z(G)\oplus \left\langle \pi\right\rangle$, and the $Z(G)$-invariants of $k\left[x_1^{\pm 1},\dots,x_p^{\pm 1}\right]$ are $$k\left[x_1^{p^2},x_1^{-p^2},\frac{x_2}{x_1},\dots,\frac{x_1}{x_p}\right]=k\left[y_2^{\pm 1},\dots,y_p^{\pm 1}\right]\left[y_1^{\pm 1}\right],$$ for $y_1=x_1^{p^2}$, $y_i=\frac{x_{i+1}}{x_i}$ for $i=2,\ldots, p$ . Observe that $\pi$ acts trivially on $y_1$ and by a $p$-root of unity on the others $y_i$, and that \[\tau: y_2\mapsto \dots\mapsto y_p\mapsto \frac{1}{y_2\dots y_p}\] and $\tau(y_1)=y_1y_2^{p^2}$. In particular, this implies that $y_1y_2^p\dots y_p^{p(p-1)}$ is $G$-invariant, so the above argument works. 

\subsubsection{}
Assume $G\cong \mathbb{Z}/p^2\rtimes \mathbb{Z}/p$. Recall the notations from Subsection \ref{second}.
 The center is generated by $\sigma^p$. The element $\sigma$ acts on any $y_i$, $i=1,\ldots,p$, by multiplication with $\zeta$, while
 \[\tau: y_2\mapsto \dots\mapsto y_p\mapsto \frac{1}{y_2\dots y_p}\]
 and $\tau(y_1)=y_1y_2^p$. Replace $y_1$ with $y_1y_2^{p-1}\dots y_{p-1}^2y_p$.
 Then $\sigma(y_1)=\zeta y_1$, and $\tau(y_1)=y_1$.
 Taking $\sigma$-invariants, $$k\left[y_1^{\pm 1},\dots,y_p^{\pm 1}\right]^{\sigma}=k\left[y_1^p,\frac{y_2}{y_1},\dots,\frac{y_p}{y_1},\text{their inverses}\right],$$
 which can be further written as $k\left[w_1^{\pm 1},\dots,w_p^{\pm 1}\right]$ for
 $w_1=y_1^p$, $w_i=\frac{y_i}{y_1},$ for $i=2,\ldots,p$.
 Observe that \[\tau: w_2\mapsto w_3 \mapsto\dots\mapsto w_p\mapsto \frac{1}{w_1\dots w_p},\]
 and thus, by replacing $w_1$ with $\frac{1}{w_1\dots w_p}$,
 we need to show that $\text{Spec }k\left[w_1^{\pm 1},\dots,w_p^{\pm 1}\right]^\tau,$
 where $\tau$ acts by $\tau: w_1\mapsto\dots \mapsto w_p\mapsto w_1$,
 is a linear scheme. This follows from Proposition \ref{33}. The same argument shows that any group of the form $\mathbb{Z}/p^s\rtimes \mathbb{Z}/p$ is mixed Tate. In particular, this means that any group $G$ of order $p^4$ and center of order $p^2$ is mixed Tate.

\subsection{}
Assume from now on that we are in the situation from Subsection \ref{432}, in which the dimension of the scheme we want to prove is linear was reduced from $p$ to $p-1$. We will explain how to obtain a $\mathbb{Z}[\tau]$-representation on $\mathbb{Z}^{p-1}$. The argument works for any $p$-group and $V$ a $p$-dimensional representation, just that in this case we will get a representation of $\mathbb{Z}[\tau]$ on $\mathbb{Z}^p$. In order to compute the $\tau$-invariants of $k\left[y_2^{\pm 1},\dots,y_p^{\pm 1}\right]^{N}$, write $N=N_1\oplus N_2$ with $N_1$ cyclic. As in the proof of the Proposition \ref{32}, we have that $$k\left[y_2^{\pm 1},\dots,y_p^{\pm 1}\right]^{N_1}=
k\left[y_2^{a_2},y_2^{a_3}y_3,\dots,y_2^{a_p}y_p,\text{their inverses}\right].$$ If we repeat the computation for $N_2$ instead of $N_1$, we find that 

$$k\left[y_2^{\pm 1},\dots,y_p^{\pm 1}\right]^{N}=
k\left[y_2^{b_2},y_2^{b_3}y_3,\dots,y_2^{b_p}y_p,\text{their inverses}\right].$$
Let $z_i:=y_2^{b_i}y_i$ for $2\leq i\leq p$. Observe that $\tau$ acts on $z_i$ in the following way: 
\begin{align*}
    \tau(z_2)&=z_2^{a_{2,2}}z_3^{a_{3,2}},\\
    \tau(z_3)&=y_2^{b_{2,3}}z_3^{b_{3,3}}z_4
\end{align*}
for some explicit integer exponents.
For any $N$-invariant $z$, the element $\tau(z)$ is also $N$-invariant because \[n\tau z=\tau n_0z=\tau z\] for some $n_0\in N$. 
In particular, $\tau(z_2)$ is $N$-invariant, so $y_2^{b_{3,2}}$ is an integer power of $z_2$. This implies that $\tau(z_3)$ is a monomial in $z_2$, $z_3$, and $z_4$, and a similar computation shows that this is true for any $2\leq k\leq p$, namely that there are integer exponents such that $$\tau(z_k)=z_2^{a_{2,k}}\dots z_{k+1}^{a_{k+1,k}}.$$
Now, we can construct a $\mathbb{Z}[\mathbb{Z}/p]\cong \mathbb{Z}[\tau]$ representation $$W:=\mathbb{Z}^{p-1}=\mathbb{Z}\log(z_2)\oplus\dots\oplus\mathbb{Z}\log(z_2)$$ by defining $$\tau(\log(z_k))=a_{2,k}\log(z_2)+\dots+a_{k+1,k}\log(z_{k+1}).$$
By a theorem of Reiner \cite{re}, the representation $W$ is isomorphic to an ideal of $\mathbb{Z}[\zeta]$, where $\zeta$ is a primitive $p$-root of unity. Chu--Kang have shown in \cite[page 687]{chu} that all such representations coming from groups of order $\leq p^3$ are generated by one element, so $W\cong \mathbb{Z}[\zeta]$. Then we can choose monomials $w_i$ in the $z_i$ on which $\tau$ acts via \[\tau:w_1\mapsto w_2\mapsto\dots\mapsto w_{p-1}\mapsto \frac{1}{w_1\dots w_{p-1}}\] and such that 
\[k\left[z_2^{\pm 1},\dots,z_p^{\pm 1}\right]=k\left[w_1^{\pm 1},\dots,w_{p-1}^{\pm 1}\right].\]
We know that $\text{Spec }k\left[w_1^{\pm 1},\dots,w_{p-1}^{\pm 1}\right]^\tau$ is a linear scheme by Proposition \ref{33}, so $V\sslash G$ is indeed a linear scheme in our case.

\section{More on mixed Tate motives of a classifying space}\label{5}

In this Section, we assume that the base field $k$ has characteristic zero. 
\subsection{} Define the triangulated category of geometrical motives \[\text{DM}_{\text{gm}}(k;R)\subset \text{DM}(k;R)\] as the smallest thick subcategory which contains all the motives $M(X)(a)$ for $X$ a separated scheme of finite type over $k$ and $a$ an integer \cite{voe}, \cite[Section 5]{to}. 
In general, the motive of a quotient stack is not a geometric motive. 
For example, for a finite non-trivial group $G$, 
the Chow groups (with $\mathbb{Z}$ coefficients) $CH^i(BG)$ are non-trivial for infinitely many values of $i$ \cite[Theorem 3.1]{yag}, and thus the motive $M(BG)\in \text{DM}(k,\mathbb{Z})$ is not geometric. 
For an explicit computation of the motive of a quotient stack, let $k(1)$ be the one-dimensional representation on which $\mathbb{G}_m$ acts with weight one. 
Observe that $\left(k(1)^{\oplus(n+1)}-0\right)/\mathbb{G}_m\cong\mathbb{P}^n$ ``approximate" the motives associated to $\mathbb{G}_m$. We thus have that 
\begin{align*}
    M(B\mathbb{G}_m)&=\bigoplus_{j\geq 0} R(j)[2j],\\
    M^c(B\mathbb{G}_m)&=\prod_{j\leq -1} R(j)[2j].
\end{align*}
None of these motives are geometric.

 Even if the motives associated to a quotient stack are not geometric motives, they exhibit some properties which resemble geometric motives. Indeed, recall that for $X$ a proper scheme, $M^c(X)\cong M(X)$, and for $X$ a smooth scheme of pure dimension $n$ over $k$, $M^c(X)\cong M(X)^*(n)[2n]$ \cite[Section 5]{to}. 
 
  Let $X=Y/G$ be a smooth quotient stack for which we can define motives $M(X)$ and $M^c(X)$, see Subsection \ref{dmdef}.
 There is an isomorphism \begin{equation}\label{isodual}
     M(X)^*\cong M^c(X)(-\dim(X))[-2\dim(X)].
 \end{equation}
 The isomorphism in \eqref{isodual} follows from the fact that the dual of a direct sum in $\text{DM}(k, R)$ is a product, so the dual of a homotopy colimit is a homotopy limit.
\\

Furthermore, the dual of a mixed Tate motive in $\text{DM}(k; R)$ is not necessarily mixed Tate.
For example, if $k$ is algebraically closed, $M:=\bigoplus_{i\in \mathbb{N}} \mathbb{Z}$ is an element of $\text{DMT}(k;\mathbb{Z})$, but its dual in $\text{DM}(k, \mathbb{Z})$ is $M^*=\prod_{i\in \mathbb{N}} \mathbb{Z}$, which is not an element of $\text{DMT}(k;\mathbb{Z})$ \cite[Corollary 4.2]{to2}. 

However, $\text{DMT}_{\text{gm}}(k;R):=\text{DMT}(k;R)\cap \text{DM}_{\text{gm}}(k; R)$ is closed under taking duals \cite[Section 5.1]{lev}. 
The main result of this Section is:

\begin{theorem}\label{thm5}
Let $G$ be a finite group, let $k$ be a field of characteristic $0$, and let $R$ be an arbitrary ring. Then $M^c(BG)\in \text{DMT}(k;R)$ is mixed Tate if and only if $M(BG)\in \text{DMT}(k;R)$ is mixed Tate.
\end{theorem}

In light of the above counterexample of a mixed Tate motive whose dual is not mixed Tate, we see that mixed Tate motives of finite groups exhibit finiteness properties. A related result \cite[Theorem 3.1]{to2} says that any scheme $X$ of finite type over a field $k$ with $M^c(X)$ mixed Tate has finitely generated Chow groups $CH^\ast(X; R)$ as $R$-modules. This implies that $CH^\ast(BG; R)$ are finitely generated over $R$, when $G$ is a finite group with $BG$ mixed Tate.

\subsection{}
We reduce the proof of Theorem \ref{thm5} to:

\begin{theorem}\label{thm6}
Let $X$ be a smooth quotient stack and let $E$ be a $\mathbb{G}_m$-bundle over $X$. Then $M(X)$ is mixed Tate if and only if $M(E)$ is mixed Tate.
\end{theorem}

Totaro has shown in \cite[Corollary 8.13]{to} that for a finite group $G$, $M^c(BG)$ is mixed Tate if and only if $M^c(GL(n)/G)$ is mixed Tate for a faithful representation $G\rightarrow GL(n)$. 
One knows that the category of geometric Tate motives $\text{DMT}_{\text{gm}}(k; R)$ is closed under taking duals, as mentioned above. Recall that for any geometric motive $X\in \text{DM}_{\text{gm}}(k; R)$, the map $X\xrightarrow{\sim} X^{**}$ is an isomorphism \cite[Lemma 5.5]{to}.  
As $GL(n)/G$ is a smooth scheme, and for any smooth scheme $S$ one has \[M(S)^*\cong M^c(S)(-\dim(S))[-2\dim(S)],\]
we see that it is enough to prove that $M(BG)$ is mixed Tate if and only if $M(GL(n)/G)$ is mixed Tate for a faithful representation $G\rightarrow GL(n)$. The strategy is to show the more general result, that for $X$ a quotient stack and $E$ a principal $GL(n)$-bundle over $X$, $M(X)$ is mixed Tate if and only if $M(E)$ is mixed Tate. The next Proposition, inspired by \cite[Lemma 7.13]{to}, shows that Theorem \ref{thm5} follows from Theorem \ref{thm6}.

\begin{lemma}
Assume that for any smooth quotient stack $X$ and any principal $\mathbb{G}_m$-bundle $F$ over $X$, $M(X)\in \text{DMT}(k; R)$ if and only if $M(F)\in \text{DMT}(k; R)$. Then, for any smooth quotient stack $X$ and any principal $GL(n)$-bundle $E$ over $X$, $M(X)\in \text{DMT}(k; R)$ if and only if $M(E)\in \text{DMT}(k; R)$.
\end{lemma}

\begin{proof}
Denote by $B$ the subgroup of upper triangular matrices in $GL(n)$.
 Then $E/B$ is an iterated projective bundle over $X$. Recall that $GL(n)$-bundles are Zariski locally trivial. We obtain the following Leray-Hirsch decomposition for motives
 $$M(E/B)\cong \bigoplus M(X)(a_j)[2a_j],$$ where $a_j$ are the dimensions of the $n!$ Bruhat cells of the flag manifold $GL(n)/B$, see also \cite[Proof of Lemma 7.13]{to}. 

Now, because $\text{DMT}(k; R)$ is closed under arbitrary direct sums, $M(X)\in \text{DMT}(k; R)$ implies $M(E/B)\in \text{DMT}(k; R)$. Conversely, $\text{DMT}(k; R)$ is thick \cite[Discussion after Lemma 5.4]{to}, so $M(E/B)\in \text{DMT}(k; R)$ implies $M(X)\in \text{DMT}(k; R)$. 

Next, let $U$ be the subgroup of strictly upper triangular matrices in $GL(n)$. Since $B/U\cong \mathbb{G}_m^n$, $E/U$ is a principal $\mathbb{G}_m^n$-bundle over $E/B$. Using the assumption about $\mathbb{G}_m$-bundles, we deduce that $M(E/U)\in \text{DMT}(k; R)$ if and only if $M(X)\in \text{DMT}(k; R)$. Finally, $U$ is an extension of copies of the additive group $\mathbb{G}_a$, so $M(E)\cong M(E/U)$, which means that $M(E)\in \text{DMT}(k; R)$ if and only if $M(X)\in \text{DMT}(k; R)$.
\end{proof}

\subsection{}
We will also need the following vanishing result:

\begin{lemma}\label{54}
If $Y$ is a smooth quasi-projective scheme, then $$\text{Hom}\,(R(i)[j],M(Y))=0,$$ for $j\leq i-2$.
\end{lemma}

\begin{proof}
Choose a smooth compactification $Z$ of $Y$ such that the complement $W:=Z\setminus Y$ is a divisor with simple normal crossings, which can be done since $k$ has characteristic $0$ \cite[Theorem 3.35]{kol}. Then, the Gysin distinguished triangle \cite[pg. 10]{voe} gives, for $c=\text{codim}\,W$:
$$M(W)\rightarrow M(Z)\rightarrow M(Y)(c)[2c]\rightarrow M(W)[1].$$ Taking the dual of this triangle we obtain, for $n=\text{dim}\,Y$:
$$M^c(W)^*(n)[2n-1]\rightarrow M(Y)\rightarrow M(Z)\rightarrow M^c(W)^*(n)[2n].$$ Both $\text{Hom}\,(R(i)[j],M(Z)[-1])$ and $\text{Hom}\,(R(i)[j],M(Z))$ are zero because $Z$ is projective. Indeed, in our case $M(Z)\cong M^c(Z)$ and $j\leq i-2$, and it is known that $\text{Hom}\,(R(i)[j],M^c(Z))=0 $ for any scheme $Z$ and any integers $i$ and $j$ with $j\leq i-1$ \cite[pg. 16]{to}. Thus, the $\text{Hom}$ long exact sequence obtained from this distinguished triangle gives that
$$\text{Hom}\,(R(i)[j],M^c(W)^*(n)[2n-1])\cong \text{Hom}\,(R(i)[j],M(Y)).$$
Observe that $W$ is proper, so $M(W)\cong M^c(W)$. Further, $$\text{Hom}\,(R(i)[j],M^c(W)^*(n)[2n-1])\cong \text{Hom}\,(M^c(W),R(n-i)[2n-1-j]).$$ Thus, it is enough to prove $$\text{Hom}\,(M^c(W),R(a)[b])=0,$$ for $b-a\geq n+1$. Further, $\text{dim}\,W<n$ and $W$ is a divisor with simple normal crossings, so there are at most $n$ divisor through any point of $W$.  To show this, we will use induction on $n$, the maximal number of divisors which pass through a given  point, and then on the number of connected components of $W$. If $n=1$ or if $W$ has only one component, then $W$ is smooth; in this case, $M(W)\cong M^c(W)$ and $M(W)^*\cong M(W)(-\text{dim}\,(W))[-2\,\text{dim}\,(W)]$. We need to show that $$\text{Hom}\left(R(i+\text{dim}\,W-n)[j+1+2\,(\text{dim}\,W-n)],M^c(W)\right)=0,$$ for $j\leq i-2$, where $i=n-a$ and $j=2n-1-b$. This follows from the vanishing property of motivic homology \[\text{Hom}\,(R(i)[j],M^c(Z))=0 \] for any scheme $Z$ and any integers $i$ and $j$ with $j\leq i-1$ \cite[page 16]{to}. In our case, $b-a\geq n+1$ is equivalent to $j\leq i-2$, and we know that $\text{dim}\,W<n$, thus $i+\text{dim}\,W-n\geq j+1+2\,(\text{dim}\,W-n)+1$.

For the general case, let $U$ a smooth connected component of $W$ and let $V$ the closure of $W\setminus U$ inside $W$. Then $V$ will be also be a divisor with simple normal crossings such that there are at most $n$ divisors passing through a given point, but will have less components than $W$. Further,  $T:=U\cap V$ will be a divisor with simple normal crossings, with at most $n-1$ divisors passing through any point. By the induction hypothesis, $\text{Hom}\,(M(T)[1], R(a)[b])=0$ for $b-a\geq n$, and $\text{Hom}\,(M(V)[1], R(a)[b])=0$ for $b-a\geq n+1$. Recall that we want to show $\text{Hom}\,(M(W)[1], R(a)[b])=0$ for $b-a\geq n+1$. For this, use the following two distinguished triangles
\begin{align*}
    M^c(U)\rightarrow M^c(W)&\rightarrow M^c(W-U)\rightarrow M^c(U)[1],\\
    M^c(T)\rightarrow M^c(V)&\rightarrow M^c(W-U)\rightarrow M^c(T)[1].
\end{align*}
From the second triangle, we get \begin{multline*}
    \text{Hom}\,(M^c(T)[1], R(a)[b])\rightarrow \text{Hom}\,(M^c(W-U),R(a)[b])\rightarrow \text{Hom}\,(M^c(V), R(a)[b])\\\rightarrow \text{Hom}\,(M^c(T), R(a)[b]).\end{multline*}
    We deduce that $\text{Hom}\,(M^c(W-U),R(a)[b])=0$ for $b-a\geq n+1$. Similarly, we can use the first triangle to deduce that $\text{Hom}\,(M^c(W),R(a)[b])=0$ for $b-a\geq n+1$.

\end{proof}

\subsection{} In this Subsection, we prove Theorem \ref{thm6}. We split its proof in a sequence of steps.

\subsubsection{} 
Let $T$ be the total space of a line bundle over $X$ such that $T-X\cong E,$ where $X\hookrightarrow T$ is embedded as the zero section.
We claim that there is a Gysin distinguished triangle \begin{equation}\label{gysin2}
M(T-X)\rightarrow M(T)\rightarrow M(X)(1)[2]\rightarrow M(T-X)[1].
\end{equation} Indeed, let $X=Y/G$ and $T=W/G$ with $Y$ smooth and $W$ an $\mathbb{A}^1$-bundle over $Y$. Consider the (smooth) approximations
\begin{align*}
    X_i&=((V_i-S_i)\times Y)/G,\\
    T_i&=((V_i-S_i)\times W)/G.
\end{align*}
Then we have the Gysin distinguished triangles \cite[Theorem 3.5.4]{voe}: 
\begin{equation}\label{gysini}
M(T_i-X_i)\rightarrow M(T_i)\rightarrow M(X_i)(1)[2]\rightarrow M(T_i-X_i)[1].
\end{equation}
The category $\text{DM}(k; R)$ is a model category with arbitrary direct sums and products \cite[Subsection 5]{to}, so it has an underlying triangulated derivator \cite[Theorem 6.11]{Cisinski}, \cite[Appendix 2, page 25]{Ke}. Thus the homotopy colimit of distinguished triangles is a distinguished triangle \cite[Corollary 11.4]{Ke}, and we thus obtain the Gysin triangle \eqref{gysin2}. Using $M(X)\cong M(T)$, the distinguished triangle \eqref{gysin2} becomes:
\begin{equation}\label{d1}
    M(E)\rightarrow M(X)\rightarrow M(X)(1)[2]\rightarrow M(E)[1].
\end{equation}
\subsubsection{}
The inclusion \[\text{DMT}(k; R)\hookrightarrow \text{DM}(k; R)\] has a right adjoint \[C:\text{DM}(k; R)\rightarrow \text{DMT}(k; R).\]  We will sometimes write $C(Z)$ instead of $C(M(Z))$ for $Z$ a quotient stack. Let $U$ be the cone of $C(E)\rightarrow M(E)$ and let $W$ be the cone of $C(X)\rightarrow M(X)$. There is a distinguished triangle 

$$U\rightarrow W\rightarrow W(1)[2]\rightarrow U[1].$$
Indeed, this triangle is induced from the triangle \eqref{d1}, the diagram
\begin{equation*}
\xymatrix{
C(E) \ar[r] \ar[d]
    & C(X) \ar[r] \ar[d]
    & C(E)(1)[2] \ar[r] \ar[d]
    & C(E)[1] \ar[d]\\
M(E) \ar[r] \ar[d]
    & M(X) \ar[r] \ar[d]
    & M(E)(1)[2] \ar[r] \ar[d]
    & M(E)[1] \ar[d]\\
U \ar[r] & W \ar[r] & W(1)[2] \ar[r] & U[1]
}
\end{equation*}
and the $3\times 3$ lemma.

\subsubsection{} 
Observe that $C(W)=0$. Indeed, $$M(X)\rightarrow C(X)\rightarrow W\rightarrow M(X)[1]$$ and, for any $i$ and $j$ integers, $$\text{Hom}(R(i)[j],M(X))\xrightarrow{\cong}\text{Hom}(R(i)[j],C(X)).$$ This implies that $W$ has trivial motivic homology groups.

Then the Tate motive $C(W)$ has trivial homology groups and so $C(W)=0$. Indeed, because $\text{Hom}(R(a)[b], C(X))=0$ and $R(a)[b]$ generate the category $\text{DMT}(k; R)$, we get that $\text{Hom}(M, C(X))=0$ for any mixed Tate motive $M$, and, in particular, that $\text{Hom }(C(X), C(X))=0$, so $C(X)=0$.

\subsubsection{}\label{disc}
We need to show $U=0$ if and only if $W=0$. If $W=0$, then it is immediate that $U=0$. Conversely, suppose $U=0$. In this case, \begin{equation}\label{aaa}
    W\cong W(1)[2].
\end{equation}
In \cite[Proposition 7.10]{dug}, Dugger and Isaksen have shown that one can compute, via a spectral sequence, the motivic homology of $X\otimes M$ from the motivic homology of $M$ and $X$, for any motive $X$ and any mixed Tate motive $M$. A related result \cite[Theorem 7.2]{to} says that if $$ C(W)\otimes C(M(Z))\xrightarrow{\cong} C(W\otimes M(Z)),$$ for any $Z$ a smooth, projective scheme, then $W$ is mixed Tate. We will use both these results in our argument below.     

The plan is the following: it is enough to show that $$ C(W)\otimes C(M(Z))\xrightarrow{\cong} C(W\otimes M(Z)),$$ for $Z$ a smooth, projective scheme. Taking into account that $C(W)\cong 0$, we will need to show that the motivic homology groups of any product $W\otimes M(Z)$ are trivial. 

We show that the motive $W$ has a vanishing property similar to the one of $M^c$ of a geometrical motive, namely that $\text{Hom}\,(R(i)[j],W)=0$ for $j\leq i-2$. Even more, we will be able to show that $\text{Hom}\,(R(i)[j],W\otimes M(Z))=0$ for $j\leq i-2$ for $Z$ a smooth projective scheme. This will imply that all the motivic homology groups of $W\otimes M(Z)$ are trivial, because $W\cong W(1)[2]$.
Consequently, we only need to show \begin{equation}\label{aa}
    \text{Hom}\,(R(i)[j],W\otimes M(Z))=0
\end{equation} for $j\leq i-2$, where $Z$ is a smooth projective scheme.

\subsubsection{}\label{discc}
First, by Proposition \ref{54}, we have that $\text{Hom}\,(R(i)[j],M(Y))=0$ for $j\leq i-2$ for a quasi-projective scheme $Y$. There is a distinguished triangle:
\begin{equation}\label{trr}
    M(X\times Z)\rightarrow C(M(X))\otimes M(Z)\rightarrow W\otimes M(Z)\rightarrow M(X\times Z)[1].\end{equation}
It is enough to show \begin{align*}
    \text{Hom}\,(R(i)[j], M(X\times Z))&=0,\\
    \text{Hom}\,(R(i)[j],C(M(X))\otimes M(Z))&=0
    \end{align*} for $j\leq i-2$.
To show that $\text{Hom}\left(R(i)[j],M(X\times Z)\right)=0$ for $j\leq i-2$, write $M(X\times Z)$ as the cone of a morphism 
\[
\bigoplus_{l\in I} M(S_l)\rightarrow \bigoplus_{l\in I} M(S_l)\to M(X\times Z)\to \left(\bigoplus_{l\in I} M(S_l)\right)[1],\]
where $S_l$ are quasi-projective schemes for $l$ in a set $I$. Because $R(i)[j]$ is a compact object inside $\text{DM}(k; R)$, we have that $$\text{Hom}\left(R(i)[j], \bigoplus_{l\in I} M(S_l)\right)=\bigoplus_{l\in I} \text{Hom}\,(R(i)[j], M(S_l))=0$$ for $j\leq i-2$. 
Finally, \begin{multline*}
    \text{Hom}\left(R(i)[j],
    \bigoplus_{l\in I} M(S_l)\right)\rightarrow \text{Hom}\,(R(i)[j],M(X\times Z))\\ \rightarrow \text{Hom}\left(R(i)[j], \left(\bigoplus_{l\in I} M(S_l)\right)[1]\right),\end{multline*} which immediately implies $\text{Hom}\left(R(i)[j],M(X\times Z)\right)=0$ for $j\leq i-2$.

\subsubsection{}
To show $\text{Hom}\left(R(i)[j],C(M(X))\otimes M(Z)\right)=0$ for $i\leq j-2$, use the motivic K\"{u}nneth spectral sequence \cite[Theorem 6.1]{to}:
$$E^{pq}_2=\text{Tor}^{H_.(k,R(\cdot))}_{-p,-q,i}\left(H_.(C(X),R(\cdot)), H_.(Z,R(\cdot))\right)\Rightarrow H_{-p-q}(C(X)\otimes Z, R(i)), $$ where $\text{Tor}_{-p,-q,i}$ denotes $(-q,i)$ bigraded piece of $\text{Tor}_{-p}$. The vanishing properties for the motivic homology of $C(M(X))$ and $M(Z)$ imply the desired result. Indeed, assume $i<0$. On the sheet $E^{pq}_2$, all nontrivial $H_.(k,R(\cdot))$ modules are concentrated in the lower left corner $j\leq i-2$, $p\leq 0$.
Every page $E^{pq}_n$ will be concentrated in the same lower left square, which implies the vanishing of motivic homology groups for $C(M(X))\otimes M(Z)$ for $j\leq i-2$. In particular, $\text{Hom}\,(R(i)[j], C(M(X))\otimes M(Z))=0$ for $j\leq i-2$. Using the triangle \eqref{trr} and the discussion in Subsection \ref{discc}, we see that \eqref{aa} holds.

\subsubsection{}
Finally, let $i$ and $j$ be arbitrary integers, and choose $a\leq i-j-2$. By \eqref{aaa} and \eqref{aa}, we have that $$\text{Hom}\,(R(i)[j], W\otimes M(Z))\cong \text{Hom}\,(R(i+a)[j+2a], W\otimes M(Z))\cong 0.$$ Thus the motivic homology of $W\otimes M(Z)$ is trivial for every smooth projective scheme $Y$. As discussed in Subsection \eqref{disc}, this implies that $W\cong 0$, and thus Theorem \ref{thm6} follows.
\\
\\


\end{document}